\documentclass[10pt,a4paper]{amsart}
\usepackage{amsfonts,amsmath,amssymb}
\usepackage{graphics}
\usepackage[dvipdfmx]{graphicx}

\allowdisplaybreaks[2]

\newtheorem{theorem}{Theorem}[section]
\newtheorem{lemma}[theorem]{Lemma}
\newtheorem{proposition}[theorem]{Proposition}
\newtheorem{corollary}[theorem]{Corollary}

\theoremstyle{definition}
\newtheorem{definition}[theorem]{Definition}
\newtheorem{example}[theorem]{Example}
\newtheorem{examples}[theorem]{Examples}

\newtheorem{remarks}[theorem]{Remarks}

\newcommand{\bbE}{\mathbb E}

\newcommand{\bbZ}{\mathbb Z}
\newcommand{\bbR}{\mathbb R}
\newcommand{\bbH}{\mathbb H}

\numberwithin{equation}{section}

\begin{document}

\title[Winding Numbers of Curves on Aspherical Surfaces]{Winding Numbers of Regular Closed Curves\\
on Aspherical Surfaces}

\author{Masayuki YAMASAKI}

\address{Department of Applied Science, Okayama University of Science, 
Okayama, Okayama 700-0005, Japan, 
E-mail: yamasaki@surgery.matrix.jp}

\begin{abstract}
We introduce based winding numbers $w_{\widetilde p}$ and free winding numbers $W_{\widetilde\gamma_0}$
of regular closed curves on surfaces with a nice euclidean or hyperbolic geometry.
We show that (1) regular closed curves with the same base point $p$ are
regularly homotopic fixing the base point if and only if they represent
the same element of the fundamental group and
have the same based winding number $w_{\widetilde p}$
and (2) two regular closed curves which are freely homotopic to a fixed closed curve $\gamma_0$
are regularly homotopic if and only if they have the same free winding number
$W_{\widetilde\gamma_0}$.
These two winding numbers are integers if the curve is orientation preserving,
and are integers modulo 2 if the curve is orientation reversing.
\end{abstract}

\maketitle


\section{Introduction}

\noindent
Whitney-Graustein theorem says that two regular closed curves on the euclidean plane
$\bbE^2$ are regularly homotopic if and only if they have the same `rotation number'
({\it a.k.a.}~`winding number')\cite{HW}.
Smale classified regular closed curves on an arbitrary surface up to regular homotopy
\cite{SS58}, 
but he did not define the winding numbers.
Winding numbers of closed curves on surfaces were introduced by Reinhart \cite{BLR}
when the surface is orientable, and later extended by Chillingworth in the non-orientable case
\cite{DRJC}.
These works are successful for surfaces with spherical or euclidean geometry, but they are too week
to give a regular homotopy classification in the case of hyperbolic surfaces.
The aim of this paper is to introduce two types of winding numbers of a regular closed
curve on a surface with nice euclidean or hyperbolic geometry.

Let us review how the classical winding number on the plane was defined.
A regular closed curve $\gamma:[a,b]\to \bbE^2$ induces an angle function
$\theta:[a,b]\to\bbR$ which measures the angle of the tangent vector from the positive direction
of the $x$-axis, and the winding number is defined to be $(\theta(b)-\theta(a))/2\pi$.
In the case of an arbitrary surface, there is no preferred direction.
The basic idea of Reinhart and Chillingworth is to use a continuous vector field $X$
on the surface to measure how the tangent vector of the curve rotates with respect to $X$.
This is fine in the euclidean case, because there are vector fileds with no singular
points on them.
In the other cases, we need to allow the vector field  to  have singular points.
Then, when a regular homotopy hits a singular point of $X$, the `winding number' jumps!
Therefore, winding numbers are defined in some quotient $\bbZ/m\bbZ$ of $\bbZ$.
Luckily, in the spherical case the $\bbZ/2\bbZ$-winding number can be used to classify
the regular closed curves.  On the sphere there are two regular homotopy classes; 
on the projective plane there are two free homotopy classes of closed curves and
in each of them there are two regular homotopy classes
which can be distinguished by the $\bbZ/2\bbZ$-winding numbers.

As long as we stick to vector fields on the surface, we will never get a winding number
which is strong enough to classify the regular homotpy classes.
Since the spherical surfaces are quite exceptional among the surfaces,
I propose to study only aspherical surfaces and use non-vanishing vector fields
on the universal covers of the surfaces.  Slightly more precisely,
we consider a riemannian surface $M$ with a nice  conformal embedding of
the universal cover $\widetilde M$ into the euclidean plane $\bbE^2$.
Such an embedding is called a {\it tame conformally euclidean structure} of $M$.
See \ref{def:ces} for the detail.
Please note that, in this paper,  a conformal map means a map
which preserves non-oriented angles at each point.
Let $\gamma:[a,b]\to M$ be a regular closed curve on such a surface $M$ with base point $p\in M$.
Fix a point $\widetilde p$ in $\widetilde M$ over $p$,
and take a lift $\widetilde\gamma:[a,b]\to \widetilde M\subset \bbE^2$
with $\widetilde\gamma(a)=\widetilde p$.
As the vector field $X$ on $\widetilde M$, we use the vector $(1,0)$ at every point
of $\widetilde M$.
Note that winding numbers can be defined not only for regular closed curves
but also for non-closed regular curves on $\bbE^2$. 
For example, a half circle has winding number $\pm 1/2$.
Such an extension is called the {\it $i$-index} of a regular curve $\widetilde\gamma$,
and is denoted $i_{\widetilde\gamma}$:
\[
i_{\widetilde\gamma}=\dfrac{\theta_{\widetilde\gamma}(b)-\theta_{\widetilde\gamma}(a)}{2\pi} \in\bbR,
\]
where $\theta_{\widetilde\gamma}$ is the angle function of $\widetilde\gamma$.
This does not depend on the choice of the angle function,
and it gives a classification under regular homotopy which fixes the base point
and the base direction (\ref{thm:main0}).
We also define the {\it $j$-index} $j_{\widetilde\gamma}$ by
\[
j_{\widetilde\gamma}=\dfrac{\theta_{\widetilde\gamma}(b)+\theta_{\widetilde\gamma}(a)}{2\pi}\in\bbR.
\]
This does depend on the choice of the angle function, but it gives a well-defined mod 2 class
in $\bbR/2\bbZ$.

Next, suppose we locally rotate $\gamma$ around the base point.
If $\gamma$ is orientation preserving, then the $i$-index $i_{\widetilde\gamma}$
of $\widetilde\gamma$ does not change; 
if $\gamma$ is orientation reversing, then the $j$-index $j_{\widetilde\gamma}$
of $\widetilde\gamma$ does not change (\ref{thm:rotation}).
When $[\gamma]=1\in\pi_1(M,p)$, $\widetilde\gamma$ is closed and 
we define the {\it based winding number} $w_{\widetilde p}(\gamma)$ of $\gamma$
to be $i_{\widetilde\gamma}\in\bbZ$.
For a curve $\gamma$ which represents a non-trivial element of $\pi_1(M,p)$,
we choose the shortest geodesic $\delta$ on $M$ 
which represents the same element of $\pi_1 (M,p)$ as $\gamma$, and take 
a lift $\widetilde\delta$ of $\delta$ which starts at a lift $\widetilde p$ of $p$.
Let $\chi$ be a suitably defined {\it external angle} (\ref{def:ext-a}) of $\delta$ at the vertex $p$.
If the curve $\gamma$ on $M$ is orientation preserving,
the {\it based winding number} $w_{\widetilde p}(\gamma)$ of $\gamma$
is defined to be $i_{\widetilde\gamma}-i_{\widetilde\delta}-\dfrac{\chi}{2\pi}$,
which turns out to be an integer.
If $\gamma$ is orientation reversing,
the {\it based winding number} $w_{\widetilde p}(\gamma)$ of $\gamma$
is defined to be the mod 2 class
$\left(j_{\widetilde\gamma}-j_{\widetilde\delta}-\dfrac{\chi}{2\pi}\right)+2\bbZ$
($\in\bbZ/2\bbZ$).
When $M$ is orientable, this number $w_{\widetilde p}(\gamma)$ does not depend on the choice of
the lift $\widetilde p$ of $p$.  When $M$ is non-orientable, the values of
$w_{\widetilde p}(\gamma)$ and $w_{\overline p}(\gamma)$ with respect to two different lifts
$\widetilde p$ and $\overline p$ of $p$ are related by the equation
$w_{\overline p}(\gamma)=\varepsilon(\widetilde p,\overline p) w_{\widetilde p}(\gamma)$,
where $\varepsilon(\widetilde p,\overline p)$ is $+1$/$-1$ if the deck-transformation
which sends $\widetilde p$ to $\overline p$ is orientation preserving/reversing
(\ref{thm:wn-lifts}).
This sign problem occurs only when $\gamma$ is orientation preserving.
In \S3,  we will prove the following:


\begin{theorem}  \label{thm:main1} Let $\gamma_1$ and $\gamma_2$ be regular closed curves on
a surface $M$ with a tame conformally euclidean structure $\widetilde M\subset \bbE^2$.
Suppose that they have the same base point $p$ and fix a lift $\widetilde p$ of $p$ in $\widetilde M$.
Then the following are equivalent.
\par \noindent
{\rm (1)} $\gamma_1$ and $\gamma_2$ are regularly homotopic fixing the base point.
\par \noindent
{\rm (2)} $[\gamma_1]=[\gamma_2]\in\pi_1(M,p)$ and
$w_{\widetilde p}(\gamma_1)=w_{\widetilde p}(\gamma_2)$.
\end{theorem}

We next study the behavior of $w_{\widetilde p}(\gamma)$
when we change $\gamma$ by a free regular homotopy in \S4.
In the case (W1) when $M$ is orientable and in the case (W2) when $M$ is non-orientable and 
$\gamma$ is orientation reversing,
we define the {\it free winding number} or simply the {\it winding number}
$W(\gamma)$ of $\gamma$ to be $w_{\widetilde p}(\gamma)$,
where $\widetilde p$ is any lift of $p$.
It is an element of $\bbZ$ in the case (W1) and an element of $\bbZ/2\bbZ$ in the case (W2).

When $M$ is non-orientable and $\gamma$ is orientation preserving, there are two cases.
The first case is (W3) in which $\gamma$ is `reversible'.
Let me explain this notion by showing two of the most typical examples.
Consider the M\"obius band $M=MB$ with the standard flat metric, and 
let $\gamma_0$ be the centerline of $M$.
Let $p$ be the base point of $\gamma_0$ and $\widetilde p$ be a lift of $p$
in the universal cover.
\begin{itemize}
\item[(1)]  Consider a regular closed curve $\gamma$ which is based at $p$ and is contained
in a small disc neighborhood of the base point $p$.  Move $\gamma$ along $\gamma_0$.
When the base point of $\gamma$ comes back to $p$, it is a `mirror image'
of the original curve, so its based winding number is $-w_{\widetilde p}(\gamma)$.
\item[(2)] Let $\gamma_0^2$ be the composite loop of $\gamma_0$ with itself.
Add a kink to $\gamma_0^2$ to get a regular closed curve $\gamma_1$ and 
add the opposite kink to $\gamma_0^2$ to get another curve $\gamma_2$.
One of the based winding numbers $w_{\widetilde p}(\gamma_1)$ and $w_{\widetilde p}(\gamma_2)$
is $+1$ and the other is $-1$.
So they are not regularly homotopic fixing the base point.
But if we rotate $\gamma_2$ along $\gamma_0$, then we obtain $\gamma_1$; 
{\it i.e.} $\gamma_1$ and $\gamma_2$ are regularly homotopic.
\end{itemize}
The curves $\gamma$, $\gamma_1$, and $\gamma_2$ above are `reversible'.
See \ref{def:rev} for a precise definition.
In the case (W3) when $\gamma$ is `reversible', we define the ({\it free}) 
{\it winding number}
 $W(\gamma)$ of $\gamma$ to be $|w_{\widetilde p}(\gamma)|\in\bbZ_+ =\{n\in\bbZ | n\geq 0\}$,
where $\widetilde p$ is any lift of $p$.

In the `non-reversible' case (W4), we need to choose a closed loop $\gamma_0$ from each 
free homotopy class of closed loops and a lift $\widetilde \gamma_0$ of $\gamma_0$.
For a regular closed curve $\gamma$ on $M$ which is freely homotopic
to $\gamma_0$, we define the ({\it free}) {\it winding number}
$W_{\widetilde\gamma_0}(\gamma)$ of $\gamma$
to be the based winding number $w_{\widetilde p}(\gamma)\in\bbZ$, where 
$\widetilde p$ is the end point of the lift of the trace of the base point by 
any free homotopy from $\gamma_0$ to $\gamma$.

The following is the main theorem of this paper.
In the cases (W1), (W2), and (W3), $W_{\widetilde\gamma_0}(\gamma_i)$ means $W(\gamma_i)$.

\begin{theorem}  \label{thm:main2} Let $M$ be a surface with a tame conformally euclidean structure
$\widetilde M\subset \bbE^2$, and $\gamma_0$ be a closed curve on $M$.
Suppose two regular closed curves $\gamma_1$ and $\gamma_2$ are freely homotopic
to $\gamma_0$ keeping the curve closed.
Then the following are equivalent.
\par \noindent
{\rm (1)} $\gamma_1$ and $\gamma_2$ are regularly homotopic.
\par \noindent
{\rm (2)} $W_{\widetilde\gamma_0}(\gamma_1)=W_{\widetilde\gamma_0}(\gamma_2)$.
\end{theorem}


\bigbreak
\section{Regular curves on the plane}

A `curve on the plane' means a parametrized curve $\gamma:[a,b]\to\bbE^2$ in this article.

\begin{definition} \label{def:regular} A $C^1$-curve $\gamma:[a,b]\to \bbE^2$ on the plane is said to be {\it regular}
if the derivative $v_\gamma(u)=\frac{d\gamma}{du}(u)$ is non-zero for every $u\in[a,b]$.
The derivative is also called the {\it velocity} of $\gamma$.
The unit vector $e_\gamma(u)=v_\gamma(u)/|v_\gamma(u)|$ is called the {\it direction} of $\gamma$.
The unit vectors $e_\gamma(a)$ and $e_\gamma(b)$ are called the {\it initial direction} and 
the {\it terminal direction} of $\gamma$, respectively.  They are also called the {\it end directions}.
\end{definition}

If $\gamma:[a,b]\to \bbE^2$ is a regular curve and $\rho:[a,b]\to [c,d]$ is an
orientation-preserving diffeomorphism ({\it i.e.} a differentiable homeomorphism
such that $\rho'(u)>0$ for all $u\in [a,b]$), then $\gamma \circ \rho^{-1}:[c,d]\to\bbE^2$
is also a regular curve with the same image.  This is called a {\it reparametrization}
of a regular curve. We identify $\gamma$ and $\gamma\circ \rho^{-1}$.  In fact,
it is customary to write
\[
(\gamma\circ\rho^{-1})(t)=\gamma(t)\qquad (c\leq t\leq d).
\]
For example, let $s$ be the {\it arc length parameter} of a regular curve 
$\gamma:[a,b]\to\bbE^2$ of length $l$; {\it i.e.}
\[
s=s(u)=\int_a^u |v_\gamma(t)|\, dt~\quad (a\leq u\leq b).
\]
Then $ds/du=|v_\gamma(u)|$ is always positive, and $s=s(u)$ defines a
diffeomorphism $s:[a,b]\to[0,l]$ with inverse $u=u(s)$. 
As I mentioned above, $\gamma(u(s))$ is denoted $\gamma(s)$, and 
$v_\gamma(s)=\frac{d\gamma}{ds}$ is a unit vector.
Here $v_\gamma(s)$ and $\frac{d\gamma}{ds}$ actually mean $v_{\gamma\circ u}(s)$
and $\frac{d\gamma}{du}\frac{du}{ds}$, respectively.

\begin{definition} \label{def:angle}
A continuous function $\theta:[a,b]\to \bbR$ is an {\it angle function} for a regular curve $\gamma$
if it is a lift of the direction map $e_\gamma:[a,b]\to S^1=\bbR/2\pi\bbZ$
with respect to the projection $\bbR\to \bbR/2\pi\bbZ$, 
or equivalently if
$e_\gamma(u)=(\cos\theta(u), \sin\theta(u))$ for each $u\in[a,b]$.  
Its value $\theta(t)$ is called the {\it angle} for the direction $e_\gamma(t)$.
\end{definition}

\begin{definition} \label{def:ij}
The {\it $i$-index} $i_\gamma$ and the {\it $j$-index} $j_\gamma$
of a regular curve $\gamma:[a,b]\to \bbE^2$ are defined as follows:
\[
i_\gamma=\frac{\theta(b)-\theta(a)}{2\pi}~(\in\bbR) \quad\text{and}\quad
j_\gamma=\frac{\theta(b)+\theta(a)}{2\pi}~(\in\bbR),
\]
where $\theta$ is any angle function for $\gamma$.
\end{definition}

The $i$-index $i_\gamma$ is obviously independent of the choice of the angle function $\theta$;
on the other hand, 
a different choice of $\theta$ changes $j_\gamma$ by an even integer; thus,
the $j$-index $j_\gamma$ is only well-defined modulo $2$.
The $i$-index is an obvious extension of the classical winding numbers of regular
closed curves on the plane to possibly non-closed regular curves.
The following is obvious.

\begin{proposition} \label{thm:ij-ab}
If $\alpha$ and $\beta$ are the angles for the directions
$e_\gamma(a)$ and $e_\gamma(b)$, respectively, then
\[
i_\gamma\in \frac{\beta-\alpha}{2\pi}+\bbZ \quad\text{and} \quad 
j_\gamma\in \frac{\beta+\alpha}{2\pi}+\bbZ.
\]
In particular, the $i$-index $i_\gamma$ is an integer if $e_\gamma(a)=e_\gamma(b)$, 
and the $j$-index $j_\gamma$ is an integer if $e_\gamma(a)$ and $e_\gamma(b)$
are reflections of each other with respect to a horizontal line.
\end{proposition}

\begin{definition} \label{def:reg-h}
A homotopy $\{\sigma_t:[a,b]\to\bbE^2\}_{0\leq t \leq 1}$ between two $C^1$-curves
$\gamma_j:[a,b]\to\bbE^2$ ($j=1,2$) is said to be {\it regular} if
\begin{enumerate}
\item $\sigma_t$ is a regular curve for each $t$ ($0\leq t \leq 1$), and
\item the velocity $v_{\sigma_t}(u)$ is continuous with respect to $t$.
\end{enumerate}
\end{definition}

\begin{proposition} \label{thm:conti}
If $\{\sigma_t\}$ is a regular homotopy,
$i_{\sigma_t}(\in\bbR)$ and $j_{\sigma_t}+2\bbZ(\in \bbR/2\bbZ)$ are continuous functions of $t$.
\end{proposition}
\begin{proof}
By choosing a suitable angle function $\varphi_t$ for $\sigma_t$, we obtain
a homotopy between $\varphi_0$ and $\varphi_1$.
\end{proof}

\begin{proposition} \label{thm:i2j}
Let $\gamma,\delta:[a,b]\to\bbE^2$ be two regular curves such that
$e_\gamma (a)=e_\delta (a)$.
Then $i_\gamma=i_\delta$ implies $j_\gamma\equiv j_\delta\mod 2$.
\end{proposition}

\begin{proof}
Since the $j$ index modulo 2 is independent of the choice of the angle function,
we may assume that the values of the angle functions 
$\theta_\gamma$ and $\theta_\delta$ of $\gamma$ and $\delta$ 
at $u=a$ are the same. 
Then $i_\gamma=i_\delta$ implies
$\theta_\gamma(b)=\theta_\delta(b)$;
therefore we have $j_\gamma\equiv j_\delta$ mod 2.
\end{proof}

The following is our main technical tool; it is a special case of a statement at the end of \S7
of \cite{SS58}.
It is also immediate from the Smale-Hirsch immersion theory (\cite{SS59}, \cite{MWH}).
Since this special case can be proved in an elementary way,
we will give a proof to make this paper more self-contained.

\begin{lemma} \label{thm:key}
Let $p$, $q$ be points in $\bbE^2$, and $\alpha$, 
$\beta\in\bbR$ be two angles.
Assume that two regular curves $\gamma_1$, $\gamma_2:[a,b]\to\bbE^2$ and their directions
$e_{\gamma_1}$, $e_{\gamma_2}$ satisfy the following boundary conditions:
\[
\gamma_j(a)=p, \quad \gamma_j(b)=q, \quad e_{\gamma_j}(a)=
(\cos\alpha,\sin\alpha) 
\quad e_{\gamma_j}(b)=(\cos\beta,\sin\beta)\qquad
(j=1,2),
\]
and let $\theta_j:[a,b]\to \bbR$ be angle functions for $\gamma_j$ $(j=1,2)$.
There exists a regular homotopy between $\gamma_1$ and $\gamma_2$ that fixes the end points
and the end directions if and only if $i_{\gamma_1} = i_{\gamma_2}$~.
\end{lemma}

\begin{proof}
The `only if' part is easy.  Let $\sigma_t$ be a regular homotopy between $\gamma_1$ and $\gamma_2$
satisfying the required conditions.
Then $i_{\sigma_t}$'s  are elements of the discrete subset
$(\beta-\alpha)/2\pi+\bbZ$ of $\bbR$.
Therefore, by \ref{thm:conti}, $i_{\sigma_t}$ is a constant function of $t$.

For the `if' part, we first note that we may assume that $p$ and $q$ are distinct points.
If $p=q$, then we can regularly homotope the curves so that they are identical near $a$ and $b$;
the restrictions of the curves to some closed interval $[a+\delta_1, b-\delta_2]$ for 
some small $\delta_1, \delta_2>0$ are non-closed curves.
Next, by using an isotopy of the plane, we may assume that the following conditions are satisfied:
\[
p=(0,0), \quad q=(d,0)=(10,0), \quad\text{and}\quad
\alpha=\beta=0~.
\]
We  may also assume that the curves $\gamma_1$ and $\gamma_2$ have
the same length $l$, because we can use a finger move to make a curve longer.
A finger move is a regular homotopy and
does not change the $i$-index of the curve.
Then reparametrize the curves by the arc length parameter $s$ ($0\leq s\leq l$).

Now we {\it squeeze} (or {\it shorten}) the curves in the direction of the origin by regular homotopies, and add a straight
tail at the terminal point.
Let us temporarily drop the subscript $j$ and let $\gamma$ represent both of $\gamma_j$'s ($j=1,2$).
For every $t$ satisfying $0< t\leq 1$, we define a curve $\tau_t$ as follows:
the first part of the curve is $t\gamma$ of length $tl$,
which starts from 0 and ends at $t(d,0)$, 
and the second part is a straight horizontal line segment of length $(1-t)d$
which starts at $t(d,0)$ and ends at $(d,0)$ (Figure 1).
The length of $\tau_t$ is $tl+(1-t)d$, so we parametrize $\tau_t$ by
$s$ ($0\leq s\leq l$) so that it has the constant speed $\dfrac{tl+(1-t)d}{l}$.

\begin{figure}[htbp]
  \begin{center}
   \includegraphics[width=140mm]{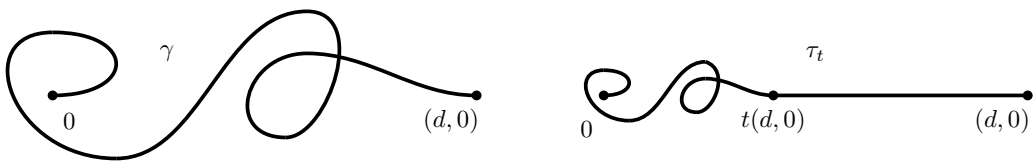}
  \end{center}
\caption{Squeezing}\label{fig:squeeze}
\end{figure}

Let us define a (small) positive number $\varepsilon$ by
$\varepsilon = {0.1}/{l}$. 
Then $\tau_t$ ($\varepsilon\leq t \leq 1$) can be thought of as a regular homotopy
between $\tau_\varepsilon$ and $\gamma$, after a reparametrization.
Throughout the regular homotopy, the end points and the end directions are fixed.  
Let us denote the curve $\tau_\varepsilon$ for $\gamma=\gamma_j$ by $\gamma_j^\varepsilon$
($j=1,2$).
They both have the same length
\[
l_\varepsilon=\varepsilon l+(1-\varepsilon)d = d+\varepsilon(l-d)~.
\]
We reparametrize the curves by the arc length parameter $s$ ($0\leq s\leq l_\varepsilon$).
Now the speed is 1, and the time it takes a point to move along the curves from the origin
to $\varepsilon(d,0)$ is $\varepsilon l=0.1$.
Here is a list of relevant numbers introduced now:
\[
0<\varepsilon l=0.1 \ll d=10 < l_\varepsilon<10.1\ll l~.
\]
The curves $\gamma_1^\varepsilon$ and $\gamma_2^\varepsilon$ have the same $i$-index $m$.
Let $\theta_1$ and $\theta_2$ be the angle functions of $\gamma_1^\varepsilon$ and 
$\gamma_2^\varepsilon$
such that 
\[
\theta_j(0)=0, \qquad \theta_j(l_\varepsilon)=2m\pi\qquad (j=1,2)~.
\]
Note that we actually have
\[
\theta_j(s)=2m\pi\qquad (0.1\leq s\leq l_\varepsilon, ~~j=1,2)~.
\]

Define a homotopy $\varphi_t$ ($0\leq t\leq 1$) between $\theta_1$ and $\theta_2$ by
\[
\varphi_t(s)=(1-t)\theta_1(s)+t\theta_2(s)~,
\]
and define a vector function $v_t(s)$ by:
\[
v_t(s)=(\cos\varphi_t(s), \sin\varphi_t(s)) \qquad  (0\leq t\leq 1, ~~0\leq s\leq l_\varepsilon).
\]
Note that, for every $t$, we have
\[
\varphi_t(0)=0, \quad \varphi_t(s)=2m\pi, \quad v_t(0)=v_t(s)=(1,0)
\qquad (0.1 \leq s\leq l_\varepsilon)~.
\]
For any $t$, the integral
\[
\int_0^s v_t(u)du\qquad (0\leq s\leq l_\varepsilon)
\]
defines a regular curve which starts from the origin with velocity $(1,0)$
and reaches its terminal point $\int_0^{l_\varepsilon} v_t(u)du$ with velocity $(1,0)$.
These determine a regular homotopy between $\gamma_1^\varepsilon$ and $\gamma_2^\varepsilon$, 
but unfortunately the terminal point may not be fixed.
Let us denote the difference of the terminal point and $(d,0)$ by $\alpha_t^\varepsilon$:
\begin{align*}
\alpha_t^\varepsilon &=\int_0^{l_\varepsilon} v_t(u)du - (d,0)
=\int_0^{0.1} v_t(u)du + \int_{0.1}^{l_\varepsilon} v_t(u)du -(d,0)\\
&=\int_0^{0.1} v_t(u)du +(l_\varepsilon-0.1,0)-(d,0)
=\int_0^{0.1} v_t(u)du +(l_\varepsilon-d-0.1,0)~.
\end{align*}
We can estimate the length of this vector as follows:
\[
|\alpha_t^\varepsilon|\leq \int_0^{0.1} |v_t(u)| du+ (0.1-(l_\varepsilon -d))
\leq 0.1+0.1=0.2~.
\] 
Recall that we are assuming that $d=10$ and note that $0.1<7<9<10=d<l_\varepsilon$.
Let $\psi:[0,l_\varepsilon] \to\bbR$ be a smooth function satisfying the following conditions:
\[
0\leq \psi(s)\leq 1~~(0\leq s\leq l_\varepsilon),\quad
\psi(s)=0~~( 0\leq s \leq 7, ~9\leq s\leq l_\varepsilon),\quad\text{and} 
\quad \int_0^{l_\varepsilon} \psi(s)ds =1~.
\]
We modify $v_t$ and define a new function $\widetilde v_t$ as follows:
\[
\widetilde v_t(s)=v_t(s)-\psi(s)\alpha_t^\varepsilon.
\]
Note that it is never 0 because $|v_t(s)|=1$ and 
$|\psi(s)\alpha_t^\varepsilon|\leq 0.2$.
Also note that $\widetilde v_t(s)=v_t(s)$ for $0\leq s\leq 7$ and
$9\leq s\leq l_\varepsilon$.  In particular, we have
\[
\widetilde v_t(0)=\widetilde v_t(l_\varepsilon)=(1,0)~.
\]

Now define a homotopy $\sigma_t$ between $\gamma_1^\varepsilon$ and $\gamma_2^\varepsilon$ by
\[
\sigma_t(s) = \int_0^s \widetilde v_t(u)du 
= \int_0^s \left( v_t(u) -\psi(u)\alpha_t^\varepsilon \right)du
= \int_0^s v_t(u)du -\left(\int_0^s \psi(u) du~\right)\alpha_t^\varepsilon~.
\]
The terminal point of the curve $\sigma_t$ can be calculated as follows:
\[
\sigma_t(l_\varepsilon) 
= \int_0^{l_\varepsilon} v_t(u)du - \left(\int_0^{l_\varepsilon} \psi(u)
du~\right)\alpha_t^\varepsilon
= \int_0^{l_\varepsilon} v_t(u)du - \alpha_t^\varepsilon
=(d,0)~.
\]
Thus $\sigma_t$ is a regular homotopy between $\gamma_1^\varepsilon$ and
$\gamma_2^\varepsilon$ which fixes the end points and the end directions.
\end{proof}


\bigbreak
\section{Regular homotopy of closed curves fixing the base point}

In this section we consider regular closed curves on a smooth surface $M$
with a good euclidean or hyperbolic geometry,
and study the regular homotopy classification of regular closed curves fixing the base point. We use the based winding number, denoted  $w_{\widetilde p}(\gamma)$,
which was already briefly described in the introduction.

We always assume that the surface $M$  has no boundary.
If a surface has a non-empty boundary and 
if a curve or a homotopy touches the boundary, we can always push it into the interior
of the surface using the collar structure near the boundary.
So the regular homotopy classification on the surface and 
the regular homotopy classification on its interior  are the same as long as 
the base point is in the interior.

We denote the universal cover of $M$ by $\widetilde M$ and denote the projection from
$\widetilde M$ to $M$ by $p_M$.
We equip $\widetilde M$ with the riemannian metric induced from that of $M$.

\begin{definition} \label{def:ces}
A {\it conformally euclidean structure} on a riemannian surface $M$ is a conformal embedding
of $\widetilde M$ into the euclidean plane $\bbE^2$.
We regard $\widetilde M$ as a subset of $\bbE^2$ using the embedding.
A conformally euclidean structure is said to be {\it tame} if the following conditions hold:
(1) any two points of $\widetilde M$ can be joined by a unique shortest geodesic
in $\widetilde M$, and 
(2) for any compact subset $K$ of the interior of $\widetilde M$, there is a regular map
$f:\bbE^2\to \widetilde M$ such that the restriction of $f$ to $K$ is the identity map.
\end{definition}

\begin{examples} \label{ex:ces}
The following riemannian surfaces have a tame conformally euclidean structure:
\begin{enumerate}
\item any complete euclidean surface $M$ ($\widetilde M= \bbE^2$),
\item any complete hyperbolic surface $M$
($\widetilde M= \bbH^2 \subset \bbE^2$), 
\end{enumerate}
where $\bbH^2$ denotes either the Poincar\'e disc model or the upper half plane model
of the hyperbolic plane.
\end{examples}

Let $p$ be a point of a surface $M$ with a conformally euclidean structure
$\widetilde M\subset \bbE^2$, 
and fix a preferred lift $\widetilde p\in\widetilde M$ of $p$.
Let $\xi:[a,b]\to M$ be a closed curve based at $p$, and let $\widetilde\xi$
be a lift of $\xi$ such that $\widetilde\xi(a)=\widetilde p$.
This determines a deck transformation $T_{\widetilde\xi}:\widetilde M\to\widetilde M$
which sends $\widetilde p=\widetilde\xi(a)$ to $\widetilde\xi(b)$.

\begin{definition} \label{def:sign}
A closed curve $\xi:[a,b]\to M$ is said to be {\it orientation preserving} or
{\it orientation reversing} according as $T_{\widetilde\xi}$ preserves the orientation of
$\widetilde M$ or not.  This does not depend on the choice of the lift $\widetilde\xi$.
Here $\widetilde M$ is given the orientation determined by the standard orientation of $\bbE^2$. 
The {\it sign} $\varepsilon_{\xi}$ of $\xi$ is defined by:
\[
\varepsilon_{\xi}=\begin{cases}
1 & \text{if $\xi$ is orientation preserving,}\\
-1 & \text{if $\xi$ is orientation reversing.}
\end{cases}
\]
The points $\widetilde\xi(a)$ and $\widetilde\xi(b)$ are said to have the
{\it same} (resp. {\it opposite}) {\it orientation} if $\xi$ is orientation 
preserving (resp. reversing), and we write 
$\varepsilon(\widetilde\xi(a), \widetilde\xi(b))=1$ (resp. $-1$).
\end{definition}

Now let $M$ be a surface with a tame conformally euclidean structure $\widetilde M\subset \bbE^2$
and  let us consider a smooth curve $\gamma:[a,b]\to M$ on $M$.  
The velocity vector $(d\gamma)_u(\frac{d}{du})$ will be denoted
$v_\gamma(u)$ as in the plane curve case.  
If $v_\gamma(u)$ is never zero, then we say $\gamma$ is {\it regular}, and
$v_\gamma(u)/|v_\gamma(u)|$ will be denoted $e_\gamma(u)$.

\begin{definition} \label{def:reg-cl}
A curve $\gamma:[a,b]\to M$ is said to be a
{\it regular closed curve} if it is reguar, closed ({\it i.e.} $\gamma(a)=\gamma(b)$), 
and $e_\gamma(a)=e_\gamma(b)$.
The point $\gamma(a)=\gamma(b)$ is called the {\it base point} of $\gamma$,
and $e_\gamma(a)=e_\gamma(b)$ is called the {\it base direction} of $\gamma$.
\end{definition} 

{\it Regular homotopies} of regular curves are defined in the obvious way.
When we consider a regular homotopy $\sigma_t$ ($0\leq t\leq 1$) 
between regular closed curves, we also require that each
$\sigma_t$ is a {\it regular closed curve} in the sense above.

As before, let us choose a point $p\in M$ and fix a point $\widetilde p\in\widetilde M$ over $p$.
The first step of the regular homotopy classification is the following.

\begin{theorem}  \label{thm:main0}
Let $\gamma_1$ and $\gamma_2$ be regular closed curves on a surface $M$
with a tame conformally euclidean structure $\widetilde M\subset \bbE^2$.
Suppose they have the same base point $p$ and the same base direction.
Let $\widetilde\gamma_1$ and $\widetilde\gamma_2$ be the lifts of $\gamma_1$ and $\gamma_2$ 
with initial point $\widetilde p$.
Then the following are equivalent.
\begin{itemize}
\item[(1)] $\gamma_1$ and $\gamma_2$ are regularly homotopic fixing the base point and the base direction.
\item[(2)] $[\gamma_1]=[\gamma_2]\in\pi_1(M,p)$ and
$i_{\widetilde \gamma_1}=i_{\widetilde\gamma_2}$.
\end{itemize}
\end{theorem}
\begin{proof}
(1)$\Rightarrow$(2): Since a regular homotopy is a homotopy, $[\gamma_1]=[\gamma_2]\in\pi_1(M,p)$
is obvious.  Since the regular homotopy lifts to a homotopy of $\widetilde\gamma_1$ and 
$\widetilde\gamma_2$ which fixes the ends and the end directions, 
$i_{\widetilde \gamma_1}=i_{\widetilde\gamma_2}$ follows from \ref{thm:key}.
\par\noindent
(2)$\Rightarrow$(1) 
By the condition $[\gamma_1]=[\gamma_2]\in \pi_1(M,p)$, the terminal points of
$\widetilde\gamma_1$ and $\widetilde\gamma_2$ are the same.
Since $\gamma_j$'s have the same base direction, $\widetilde\gamma_j$'s have the same end directions.
Now we obtain a desired regular homotopy between $\widetilde\gamma_j$'s in $\bbE^2$ by \ref{thm:key}.
By the tameness of the structure, we can squeeze the regular homotopy into $\widetilde M$ fixing
$\widetilde\gamma_j$'s.  
Now project this regular homotopy down to $M$ to get a desired regular homotopy on $M$.
\end{proof}

We next study regular homotopies which fix the base point but not necessarily fix the base direction.
The following is a key fact about $i$-indices and mod 2 $j$-indices.

\begin{proposition} \label{thm:rotation}
Let $\gamma:[a,b]\to M$ be a regular closed curve based at $p$, and $\widetilde\gamma$
be a lift of $\gamma$ whose initial point is a lift $\widetilde p$ of $p$.
\begin{itemize}
\item[(1)] When $\gamma$ is orientation preserving, a local rotation of $\gamma$ around the base point
does not change the $i$-index $i_{\widetilde\gamma}$ $(\in\bbR)$ of $\widetilde\gamma$.
\item[(2)] When $\gamma$ is orientation reversing, a local rotation of $\gamma$ around the base point
does not change the mod 2 $j$-index $j_{\widetilde\gamma}+2\bbZ$ $(\in \bbR/2\bbZ)$ 
of $\widetilde\gamma$.
\end{itemize}
\end{proposition}
\begin{proof}
Suppose we locally rotate $\gamma$ around the base point so that
$\widetilde\gamma$ rotates aroud $\widetilde\gamma(a)$ by an angle $\alpha$.
Then the other end of $\widetilde\gamma$ rotates around $\widetilde\gamma(b)$
by an angle $\varepsilon_\gamma\alpha$, because the deck transformation $T$ of $\widetilde M$
which sends $\widetilde\gamma(a)$ to $\widetilde\gamma(b)$ is an isometry and
$T$ is orientation preserving (resp. orientation reversing) 
if $\gamma$ is orientation preserving (resp. orientation reversing).
\end{proof}

So we would like to use the $i$-index (for an orientation preserving curve) and
the $j$-index (for an orientation reversing curve) to define the 
{\it based winding number}.
The $i$-index is an integer if the curve is null-homotopic;
but these indices may not be integers in general.

So let us take a regular closed curve $\gamma$ based at $p$ which is not null-homotopic.
Then $\widetilde p=\widetilde\gamma(a)$ and $\widetilde\gamma(b)$ are distinct points.
We first pick up the unique geodesic $\widetilde\delta:[a,b]\to\widetilde M$
connecting these two points in $\widetilde M$ with the riemannian metric induced from that of $M$.
This depends only on the homotopy class $\xi$ of $\gamma$ in $\pi_1(M,p)$ and 
on the choice of $\widetilde p$.
The geodesic $p_M\circ\widetilde\delta$ on $M$ will be denoted $\delta$.
This $\delta$ is the shortest geodesic based at $p$ that represents $\xi$ in $\pi_1(M,p)$,
and can be regarded as the simplest closed curve on $M$ that represents $\xi$.
So we would like to compare $\gamma$ with $\delta$.
Note that $\delta$ may have a corner at the base point $p$; so, we need to take
the external angle of $\delta$ at the vertex $p$ into consideration.
Fortunately, the internal angle of $\delta$ at $p$ is not 0, because
a point and a direction determines a geodesic which passes the given point and 
has the given direction there;
therefore,
the external angle cannot be equal to $\pm \pi$.
Since we are not assuming that $M$ is orientable, we fix the orientation convention
as follows.

\begin{definition} \label{def:ext-a}
Let $\widetilde\delta$ be as above, and let $\widehat\delta:
[a,b]\to\widetilde M$ be the lift of $\delta$ whose initial point $\widehat\delta(a)$
is equal to the terminal point $\widehat{p}=\widetilde\delta(b)$ of $\widetilde\delta$. 
The {\it external angle $\chi_{\delta,\widetilde p}$ of $\delta$ with respect to} $\widetilde p$
is defined to be the unique real number satisfying
\[
\chi_{\delta,\widetilde p}\equiv 
\theta_{\widehat\delta}(a) - \theta_{\widetilde\delta}(b) \mod 2\pi
\quad\text{and}\quad -\pi< \chi_{\delta,\widetilde p} <\pi.
\]
\end{definition}

Next, define two real numbers $\alpha$ and $\beta$ as follows:
\[
\alpha = \theta_{\widetilde\gamma}(a)- \theta_{\widetilde\delta}(a), \quad
\beta = \theta_{\widetilde\gamma}(b)- \theta_{\widetilde\delta}(b) .
\]
Then we have the following identities.

\begin{proposition} \label{thm:ije-ab}
$i_{\widetilde\gamma}-i_{\widetilde\delta}=\dfrac{\beta-\alpha}{2\pi},
~j_{\widetilde\gamma}-j_{\widetilde\delta}=\dfrac{\beta+\alpha}{2\pi}, 
~\chi_{\delta,\widetilde{p}}\equiv \beta-\varepsilon_\gamma\alpha$ {\rm (mod $2\pi$)}.
\end{proposition}

\begin{proof}
The first two are obvious.  The third identity (mod $2\pi$) can be proved as follows:
\begin{align*}
\chi_{\delta,\widetilde p}\equiv\theta_{\widehat\delta}(a) - \theta_{\widetilde\delta}(b)
&= (\theta_{\widehat\gamma}(a)-\theta_{\widetilde\delta}(b))
-(\theta_{\widehat\gamma}(a)-\theta_{\widehat\delta}(a))\\
&\equiv (\theta_{\widehat\gamma}(a)-\theta_{\widetilde\delta}(b))
-\varepsilon_\gamma(\theta_{\widetilde\gamma}(a)-\theta_{\widetilde\delta}(a))
=\beta-\varepsilon_\gamma\alpha~.
\end{align*}
\end{proof}

Thus, if $\gamma$ is homotopically non-trivial, the following is true:

\begin{corollary} \label{thm:ij-int}If $\varepsilon_\gamma=1$, 
then $i_{\widetilde\gamma}-i_{\widetilde\delta}-\dfrac{\chi_{\delta,\widetilde{p}}}{2\pi}$
is an integer.
If $\varepsilon_\gamma=-1$,
then $j_{\widetilde\gamma}-j_{\widetilde\delta}-\dfrac{\chi_{\delta,\widetilde{p}}}{2\pi}$
is an integer.
\end{corollary}

\begin{definition} \label{def:wn}
The {\it based winding number} $w_{\widetilde p}(\gamma)$ of $\gamma$
is defined as follows:
\[
w_{\widetilde p}(\gamma)=\begin{cases}
i_{\widetilde\gamma}\in\bbZ & \text{if $\gamma$ is null-homotopic,}\\
i_{\widetilde\gamma}-i_{\widetilde\delta} -\dfrac{\chi_{\delta,\widetilde p}}{2\pi}\in\bbZ
& \text{if $\varepsilon_\gamma=1$ and $\gamma$ is not null-homotopic, }\\
j_{\widetilde\gamma}-j_{\widetilde\delta} -\dfrac{\chi_{\delta,\widetilde p}}{2\pi} + 2\bbZ \in\bbZ/2\bbZ
& \text{if $\varepsilon_\gamma=-1$~.}
\end{cases}
\]
\end{definition}

\noindent
{\bf Proof of Theorem \ref{thm:main1}.}

\noindent
(1) $\Rightarrow$ (2): 
$[\gamma_1]=[\gamma_2]\in\pi_1(M,p)$ is obvious.
Let $\widetilde\gamma_j$ ($j=1,2$) be the lift of $\gamma_j$ whose initial point is $\widetilde p$.
Then $\widetilde\gamma_j$'s are regularly homotopic fixing the base point.
By such a regular homotopy, the $i$-indices and $j$-indices change continuously.
When $\gamma_j$'s are homotopically non-trivial, 
the geodesic $\delta$ is unchanged;
therefore, $w_{\widetilde p}(\gamma)$ changes continuously.
But the value set is discrete, so it does not change.

\noindent
(2) $\Rightarrow$ (1):
Let $\widetilde\gamma_j$ ($j=1,2$) be the lift of $\gamma_j$ whose initial point is $\widetilde p$.
First assume that $\gamma_1$ and $\gamma_2$ are null-homotopic.
Then $\widetilde\gamma_j$'s are regular closed curves based at $\widetilde p$.
By a local rotation of $\gamma_2$ around the base point,
we may assume that $\gamma_j$'s have the same direction at the base point.
Note that a local rotation is a regular homotopy, 
and that, by \ref{thm:rotation}, it does not change $i_{\widetilde\gamma_2}$.
Now by \ref{thm:main0}, we obtain a desired
regular homotopy between $\gamma_1$ and $\gamma_2$.

Next assume that $\gamma_1$ and $\gamma_2$ are not null-homotopic.
By the condition $[\gamma_1]=[\gamma_2]\in \pi_1(M,p)$, the terminal points of
$\widetilde\gamma_1$ and $\widetilde\gamma_2$ are the same.
A local rotation of $\gamma_j$ around $p$ is a regular homotopy, 
so it does not change the situation,
and we may assume that the initial directions of $\widetilde\gamma_1$ and
$\widetilde\gamma_2$ are both $(1,0)$, and 
the terminal directions of $\widetilde\gamma_1$ and $\widetilde\gamma_2$ are the same.
Let $\widehat p$ denote the common terminal point of $\widetilde\gamma_j$'s.
If $\gamma_j$'s are orientation preserving, then 
$w_{\widetilde p}(\gamma_1)=w_{\widetilde p}(\gamma_2)$
implies that $i_{\widetilde\gamma_1}=i_{\widetilde\gamma_2}$ and the result follows 
from \ref{thm:main0}.
If $\gamma_j$'s are orientation reversing, then $w_{\widetilde p}(\gamma_1)=w_{\widetilde p}(\gamma_2)$
implies that $j_{\widetilde\gamma_1}\equiv j_{\widetilde\gamma_2} \mod 2$.
Choose the angle function $\theta_{\widetilde\gamma_j}$
of $\widetilde\gamma_j$ satisfying $\theta_{\widetilde\gamma_j}(a)=0$ ($j=1,2$).
Then we have
\[
i_{\widetilde\gamma_1}=j_{\widetilde\gamma_1}=\frac{\theta_{\widetilde\gamma_1}(b)}{2\pi}, \qquad
i_{\widetilde\gamma_2}=j_{\widetilde\gamma_2}=\frac{\theta_{\widetilde\gamma_2}(b)}{2\pi}, \qquad
\text{and} \qquad j_{\widetilde\gamma_1}\equiv j_{\widetilde\gamma_2}\mod 2.
\]
Therefore, there exists an integer $k$ such that
\[
\theta_{\widetilde\gamma_1}(b)=\theta_{\widetilde\gamma_2}(b)+4k\pi.
\]
Now let us locally rotate $\gamma_2$ around $p$ by an angle of magnitude $2|k|\pi$
to get a new curve $\gamma_2'$.
This induces local rotations of $\widetilde\gamma_2$ around $\widetilde p$ and $\widehat p$
by angle of magnitude $2|k|\pi$, but the sign of the angles are opposite.
We choose the rotation so that the angle function $\theta_{\widetilde\gamma'_2}$ satisfies
\[
\theta_{\widetilde\gamma'_2}(b)=\theta_{\widetilde\gamma_2}(b)+2k\pi, \qquad
\theta_{\widetilde\gamma'_2}(a)=\theta_{\widetilde\gamma_2}(a)-2k\pi = -2k\pi.
\]
Then
\[
i_{\widetilde\gamma'_2}=\frac{(\theta_{\widetilde\gamma_2}(b)+2k\pi)-(-2k\pi)}{2\pi}=
\frac{\theta_{\widetilde\gamma_2}(b)+4k\pi}{2\pi}=\frac{\theta_{\widetilde\gamma_1}(b)}{2\pi}=
i_{\widetilde\gamma_1}
\]
Now we have a regular homotopy between $\gamma_1$ and $\gamma_2'$ by \ref{thm:main0}.
\hfill$\square$

We close this section by proving the following formula.

\begin{proposition} \label{thm:wn-lifts}
Let $\widetilde p$ and $\overline{p}$ be two lifts of $p$.
Then $w_{\overline{p}}(\gamma)=\varepsilon(\widetilde p,\overline p)
w_{\widetilde p}(\gamma)$.
\end{proposition}

\begin{proof}
Let $\widetilde\gamma$ and $\overline\gamma$ be the lifts of $\gamma$
whose initial points are $\widetilde p$ and $\overline p$, respectiely.
Note that the deck transformation $T$ that sends $\widetilde p$ to $\overline p$
maps $\widetilde\gamma$ to $\overline\gamma$ isometrically with
respect to the riemannian metric coming from that of $M$, but 
only conformally with respect to the euclidean metric.

First assume that $\gamma$ is null-homotopic.
The two curves $\widetilde\gamma$ and $\overline\gamma$ are both regular closed curves.
If we squeeze $\widetilde\gamma$ into a small neighborhood of $\widetilde p$
by a regular homotopy fixing the base and the base direction,
then $\gamma$ and $\overline\gamma$ also deform into small neighborhoods of the bases
in similar ways.
This does not change $i_{\widetilde\gamma}$ and $i_{\overline\gamma}$.
If we keep squeezing $\widetilde\gamma$ smaller and smaller,
the geometry coming from $\widetilde M$ around $\widetilde\gamma$ and
$\overline\gamma$ becomes almost euclidean, and the deck transformation $T$
sends $\widetilde\gamma$ to $\overline\gamma$ keeping the shape 
almost the same with respect to the euclidean metric, after rescaling.
So, we have $i_{\overline\gamma}=\varepsilon(\widetilde p,\overline p)
i_{\widetilde\gamma}$, as desired.
We can make this argument more rigorous by expanding $\overline\gamma$ from $\overline p$
by the factor of $1/s$ 
when we shrink $\widetilde\gamma$ toward $\widetilde p$ by the factor of $s$ ($0<s\leq 1$), 
so that we can actually see the changes as follows.
Define a homotopy $H:\widetilde M\times[0,1]\to\bbE^2$ of $T$ by
\[
H(x,s)=\begin{cases}
\dfrac{1}{s}(T(s x+(1-s)\widetilde p)-q) +q & (0<s\leqq 1)\\
A(x-\widetilde p)+q=:g(x) & (s=0)
\end{cases}
\qquad \text{($x\in \widetilde M$).}
\]
Here $A$ denotes the Jacobi matrix $(DT)_{\widetilde p}$ of $T$ at $\widetilde p$, and
it is a multiple of an orthogonal matrix by a positive number, because $T$ is conformal
with respect to the euclidean metric.
This defines a smooth isotopy between $T$ and $g$ described above.
See Milnor's argument on p.34 of \cite{M}.
Therefore we have
\[
i_{\overline\gamma}=i_{T\circ\widetilde\gamma}=i_{g\circ\widetilde\gamma}=
\varepsilon(\widetilde p, \overline p)i_{\widetilde\gamma}.
\]

Next assume that $\gamma$ is not null-homotopic.
Let $\widetilde\delta$ be the shortest geodesic in $\widetilde M$ that connects 
$\widetilde\gamma(a)$ and $\widetilde\gamma(b)$.
Then $\overline\delta=T\circ\widetilde\delta$ is the shortest geodesic in $\widetilde M$
connecting $\overline\gamma(a)$ and $\overline\gamma(b)$.

We prove the following three identities which confirm the proposition.
\begin{equation} \label{eqn:*}
i_{\overline\gamma}-i_{\overline\delta}=\varepsilon(\widetilde p,\overline p)
(i_{\widetilde\gamma}-i_{\widetilde\delta}), 
\qquad 
j_{\overline\gamma}-j_{\overline\delta}\overset{\mathrm{mod}\ 2}{\equiv}
\varepsilon(\widetilde p,\overline p)
(j_{\widetilde\gamma}-j_{\widetilde\delta}), 
\qquad 
\chi_{\delta,\overline p}=\varepsilon(\widetilde p,\overline p)
\chi_{\delta,\widetilde p}. \tag{*}
\end{equation}
Recall from \ref{thm:ije-ab} that $i_{\widetilde\gamma}-i_{\widetilde\delta}$
is equal to $\frac{\beta-\alpha}{2\pi}$, where 
$\alpha=\theta_{\widetilde\gamma}(a)-\theta_{\widetilde\delta}(a)$
and $\beta=\theta_{\widetilde\gamma}(b)-\theta_{\widetilde\delta}(b)$.
Make a local rotation of $\widetilde\delta$ around the initial point
$\widetilde p$ by the angle $\alpha$, and another local rotation
around the terminal point $\widehat p$ by the angle $\beta$
to obtain a new curve $\widetilde{\delta}'$.
Now $\widetilde{\delta}'$
and $\widetilde\gamma$ have the same ends and the same end directions.
Choose an angle function $\theta'$ for $\widetilde{\delta}'$ which
coincides with $\theta_{\widetilde\delta}$ except near the ends.
Then we have
\[
\theta'(b)=\theta_{\widetilde\delta}(b)+\beta, \qquad
\theta'(a)=\theta_{\widetilde\delta}(a)+\alpha.
\]
From these, we obtain
\[
\theta'(b)-\theta'(a)=(\theta_{\widetilde\delta}(b)-\theta_{\widetilde\delta}(a))
+(\beta-\alpha) \quad\text{\it i.e.}\quad 
i_{\widetilde\delta'}-i_{\widetilde\delta}=\frac{\beta-\alpha}{2\pi}.
\]
Therefore, we have $i_{\widetilde\gamma}=i_{\widetilde\delta'}$.
By \ref{thm:key}, $\widetilde\gamma$ and $\widetilde\delta'$ are regularly
homotopic fixing the ends and the end directions.
Composing this regular homotopy with $T$, we obtain a regular homotopy
between $\overline\gamma$ and $T\circ\widetilde\delta'$
fixing the ends and the end directions; hence, we have 
$i_{\overline\gamma}=i_{T\circ\widetilde\delta'}$ again by \ref{thm:key}.
Note that $\overline\delta=T\circ\widetilde\delta$ and $T\circ\widetilde\delta'$ are the same
except near the ends and that
$T\circ\widetilde\delta'$ is obtained by rotating 
$\overline\delta=T\circ\widetilde\delta$ by the angle 
$\varepsilon(\widetilde p,\overline p)\alpha$ around the initial point
$\overline\delta(a)$ and by the angle
$\varepsilon(\widetilde p,\overline p)\beta$ around the terminal point
$\overline\delta(b)$.
Thus, we have the first identity:
\[
i_{\overline\gamma}=i_{T\circ\widetilde\delta'}=i_{\overline\delta}+
\varepsilon(\widetilde p,\overline p)\frac{\beta-\alpha}{2\pi}
=i_{\overline\delta}+
\varepsilon(\widetilde p,\overline p)(i_{\widetilde\gamma}-i_{\widetilde\delta}).
\]
Now, using \ref{thm:i2j} and \ref{thm:ije-ab}, we obtain the second identity:
\[
j_{\overline\gamma}\overset{\mathrm{mod}\ 2}{\equiv}
j_{T\circ\widetilde\delta'}=j_{\overline\delta}+
\varepsilon(\widetilde p,\overline p)\frac{\beta+\alpha}{2\pi}
=j_{\overline\delta}+
\varepsilon(\widetilde p,\overline p)(j_{\widetilde\gamma}-j_{\widetilde\delta}).
\]
Similarly, the external angles $\chi_{\delta,\widetilde p}$ and $\chi_{\delta,\overline p}$
can be computed using  the local data on the angles around $\widetilde\gamma(b)$
and $\overline\gamma(b)$; 
therefore, we have the third identity. This finishes the proof of the identities (\ref{eqn:*})
and the proof of \ref{thm:wn-lifts}.
\end{proof}

\begin{corollary} \label{thm:indep}
If $M$ is orientable or if $\gamma$ is orientation reversing,
then $w_{\widetilde p}(\gamma)$ is independent of the choice of the lift $\widetilde p$
of the base point $p$ of $\gamma$.
\end{corollary}

\bigbreak
\section{Free regular homotopy of closed curves}

In this section, we assume that $M$ is a riemannian 
surface with a tame conformally euclidean structure, and
study the free regular homotopy classification of regular closed curves on $M$.
We first fix a closed curve $\gamma_0$ in $M$ based at $p_0$, and,
for a regular closed curve $\gamma$ which is freely homotopic to $\gamma_0$
as closed curves, we define the ({\it free}) {\it winding number
$W_{\widetilde\gamma_0}(\gamma)$ of $\gamma$},
which is used for the free regular homotopy classification \ref{thm:main2}.

We sometimes need to compose consecutive curves/homotopies.
Usually this needs reparametrization of the curves/homotopies,
but we do not mention such reparametrizations.  

Let us first look at the effect of a (free) regular homotopy of regular closed curves
on the based winding number $w_{\widetilde p}(\gamma)$.

\begin{proposition} \label{thm:wn-reg-h}
Suppose that $\{\gamma_t|1\leq t\leq 2\}$ is a regular homotopy between regular closed curves 
$\gamma_1$ and $\gamma_2$ on $M$ based at $p_1$ and $p_2$,
and let $\{\widetilde\gamma_t\}$ be a lift of the homotopy $\{\gamma_t\}$ in $\widetilde M$.
Let $\widetilde p_1$ and $\overline{p}_2$ be the initial points of 
$\widetilde\gamma_1$ and $\widetilde\gamma_2$, respectively.
Then the identity
$w_{\widetilde p_1}(\gamma_1)=w_{\overline{p}_2}(\gamma_2)$ holds.
Furthermore, if $p_1=p_2$, then 
$w_{\widetilde p_1}(\gamma_1)=
\varepsilon(\widetilde p_1,\widetilde p_2)w_{\widetilde p_1}(\gamma_2)$ holds.
\end{proposition}

\begin{proof}
First let us consider the case when $\gamma_j$'s are null-homotopic.
In this case, $w_{\widetilde p}(\gamma)$ is the integer $i_{\widetilde\gamma}$,
where $\widetilde\gamma$ is a lift of $\gamma$ with initial point $\widetilde p$.
The regular homotopy of $\widetilde\gamma$ changes $i_{\widetilde p}$ continuously
and the set of integers are discrete; therefore, it must be constant.

The case when $\gamma_j$'s are not null-homotopic can be handled similarly,
because $i_{\widetilde\gamma}-i_{\widetilde\delta}$, $j_{\widetilde\gamma}-j_{\widetilde\delta}$, 
and the external angle $\chi_{\delta,\widetilde p}$ change continuously.

Now \ref{thm:wn-lifts} implies the second part immediately.
\end{proof}

Note that $\overline{p}_2$ is determined by $\widetilde p_1$ and the regular homotopy.

\begin{definition} \label{def:W1-W2}
Let $\gamma$ be a regular closed curve on $M$.
\begin{itemize}
\item[(W1)] If $M$ is orientable, then $W(\gamma)\in\bbZ$ is defined to be $w_{\widetilde p}(\gamma)$, 
where $\widetilde p$ is any lift of the base point $p$ of $\gamma$.
\item[(W2)] If $M$ is non-orientable and $\gamma$ is orientation reversing, then
$W(\gamma)\in\bbZ/2\bbZ$ is defined to be $w_{\widetilde p}(\gamma)$, where $\widetilde p$ is
any lift of the base point $p$ of $\gamma$.
\end{itemize}
\end{definition}

Note that $W(\gamma)$ is independent of the choice of $\widetilde p$ by \ref{thm:wn-lifts},
and that it is a free regular homotopy invariant by
\ref{thm:wn-lifts} and \ref{thm:wn-reg-h}, in these cases.

Before continuing the definition, let us look at a very important class of regular homotopies.

\begin{example} \label{ex:finger}
Let $\gamma$ be a regular closed curve with base point $p$, and $\xi$
be a regular curve with initial point $p$ such that the initial direction
of $\delta$ and the base direction of $\gamma$ are linearly independent.
A finger move of the base point of $\gamma$ along the curve $\xi$ is a typical
example of a regular homotopy.
When the curve $\xi$ is closed, the homotopical effect of
the finger move of this type is the
conjugation $[\xi]^{-1}[\gamma][\xi]\in\pi_1(M,p)$ of the class $[\gamma]$.
\end{example}

Suppose $M$ is non-orientable and $\gamma$ is a regular closed curve on $M$ based at $p$.
Take an orientation reversing regular loop $\xi$ based at $p$, and
perform a finger move of $\gamma$ along $\xi$ to obtain a regular closed curve $\gamma'$.
By \ref{thm:wn-reg-h}, we have $w_{\widetilde p}(\gamma')=-w_{\widetilde p}(\gamma)$.
We are particularly interested in the case where $[\gamma]=[\gamma']\in\pi_1(M,p)$
({\it i.e.} $[\gamma]$ and $[\xi]$ commute in $\pi_1(M,p)$).
This is the reason we introduce the notion of `reversibility':

\begin{definition} \label{def:rev}
(1) Let $G$ be a group with a fixed non-trivial homomorphism $w:G\to\{\pm1\}$.
An element $g\in G$ is said to be {\it reversible} if there is an element $h$ of the 
centralizer $C_G(g)$ of $g$ in $G$ such that $w(h)=-1$.
\\
(2) Let $M$ be a non-orientable surface.  A loop $\gamma$ on $M$ based at $p\in M$ is said to be
{\it reversible} if the element $[\gamma]\in\pi_1(M,p)$ is reversible
with respect to the orientation homomorphism $w:\pi_1(M,p)\to\{\pm1\}$.
If $\xi$ is an orientation reversing closed curve on $M$ based at $p$ such that
$[\xi^*\gamma\xi]=[\gamma]\in \pi_1(M,p)$ ({\it i.e.} $[\xi]\in
C_{\pi_1(M,p)}([\gamma])$), then we say that $\xi$ {\it reverses} $\gamma$.
Here $\xi^*$ denotes the inverse of the path $\xi$.
\end{definition} 

\begin{remarks} \label{rem:rev} (1) Let $(G,w)$ be as above.  
If $\varphi:G\to G$ is an isomorphism satisfying $w\circ\varphi=w$,
then, for any $g\in G$, $\varphi(g)$ is reversible if and only if $g$ is reversible.
For example, conjugation by an element of $G$ determines such an isomorphism; therefore,
reversibility is preserved by conjugation.
\\
(2) Let $G$ be as above.
Then, any element which is a power of an element $h$ with $w(h)=-1$ is reversible,
because $h$ and its power $h^n$ commute.
\\
(3) Let $M$ be a Klein bottle.  Its fundamental group
has a presentation $\langle a,b|abab^{-1}=1\rangle$, where $w(a)=1$ and $w(b)=-1$.
Any element can be written uniquely in the form $a^mb^n$ ($m,n\in\bbZ$).
Such an element is orientation preserving if and only if $n$ is even.
When $n$ is even, $a^mb^n$ is reversible if and only if $m=0$.
\\
(4) Let $M$ be a non-compact non-orientable complete hyperbolic surface
({\it e.g.} a punctured Klein bottle).
Toshio Sumi pointed out the following to the author: 
Its fundamental group $G=\pi_1(M,*)$ is a free group generated by a set $S$. 
If $g\in G$ is non-trivial and orientation-preserving, then 
its centralizer $C_G(g)$ is infinite cyclic (\cite{J}, p.11).
Since $g$ itself is an element of  $C_G(g)$,
$g$ is reversible if and only if it is a power of an orientation reversing element.
Write $g$ in the reduced form
$s_1^{\varepsilon_1}s_2^{\varepsilon_2}\dots s_n^{\varepsilon_n}$,
where $s_i$'s are elements of the set $S$ and $\varepsilon_i=\pm1$.
If this has the form $s_1^{\varepsilon_1}g' s_1^{-\varepsilon_1}$, then 
the reversibility of $g$ coincides with that of $g'$; therefore, we only need to consider 
the case when $s_1^{-\varepsilon_1}\ne s_n^{\varepsilon_n}$.  Then,
if it is a power of some element $h$, the reduced form of $h$ must start with
$s_1^{\varepsilon_1}$ and end with $s_n^{\varepsilon_1}$, so 
it is not difficult to determine the reversibility.
\\
(5) Let $M$ be a closed hyperbolic surface. For any non-trivial orientation
preserving element $g$ of $G=\pi_1(M,*)$, 
its centralizer $C_G(g)$ is infinite cyclic,
because $g$ preserves a geodesic $\gamma_g$ of the hyperbolic plane $\bbH^2$ and the action of
$C_G(g)$ on $\bbH^2$ restricts to a discrete action on $\gamma_g$.
See the similar argument given at the end of \S1.1 of \cite{FM} for the orientable surface case.
Therefore, $g$ is reversible if and only if the generator of $C_G(g)$ is orientation
reversing. Is there an algorithm to find the generator of $C_G(g)$ of a given element $g$?
\end{remarks}

\begin{proposition} \label{thm:rev-ht}
Reversibility of a closed curve depends only on its free homotopy class.
\end{proposition}
\begin{proof}
Suppose there is a free homotopy $H$ between two closed curves $\gamma$ and $\gamma'$
which keeps the curve closed, and let $\eta$ be the induced curve from the base point $p$
of $\gamma$ to the base point $p'$ of $\gamma'$.
Then, $H$ can be thought of as a homotopy between $\eta^*\gamma\eta$ and $\gamma'$
which fixes the base point $p'$.
Assume that $\gamma$ is reversible and let $\xi$ be an orientation reversing closed
curve based at $p$ such that there is a homotopy $K$ between $\gamma$
and $\xi^*\gamma\xi$ which keeps the base point fixed.
Consider the closed curve $\eta^*\xi\eta$ based at $p'$.
Let $\widetilde p$ be a lift of $p$ in $\widetilde M$, and let $\widetilde\eta$ denote
the lift of $\eta$ with initial point $\widetilde p$.
Then it can be easily checked that the deck transformation $T_{\widetilde\xi}$
is the same as the deck transformation 
$T_{{\widetilde\eta}^*\widetilde\xi\widetilde\eta}$; therefore, 
$\eta^*\xi\eta$ is orientation reversing.
Now we can show that $\gamma'$ is also reversible, because the `conjugate'
$(\eta^*\xi^*\eta)\gamma'(\eta^*\xi\eta)$ of $\gamma'$ by the orientation
reversing closed curve $\eta^*\xi\eta$ is homotopic to $\gamma'$ fixing the base point:
$(\eta^*\xi^*\eta)\gamma'(\eta^*\xi\eta) \overset{H}{\simeq}
\eta^*\xi^*\gamma\xi\eta \overset{K}{\simeq}
\eta^*\gamma\eta \overset{H}{\simeq} \gamma'$.
\end{proof}

\begin{definition}[conitinued from \ref{def:W1-W2}] \label{def:W3}
Let $\gamma$ be a regular closed curve on $M$.
\begin{itemize}
\item[(W3)] If $M$ is non-orientable, and $\gamma$ is orientation preserving and reversible, then
$W(\gamma)\in\bbZ_+=\{n\in\bbZ|n\geq 0\}$ is defined to be $|w_{\widetilde p}(\gamma)|$, 
where $\widetilde p$ is any lift of the base point $p$ of $\gamma$.
\end{itemize}
\end{definition}

$W(\gamma)$ is independent of the choice of $\widetilde p$ by \ref{thm:wn-lifts},
and that it is a free regular homotopy invariant by \ref{thm:wn-lifts} and \ref{thm:wn-reg-h},
in this case.

\begin{definition} \label{def:fwn}
Let $M$ be a surface with a tame conformally euclidean sturcture, $\gamma_0$ be 
a closed curve on $M$ based at $p_0\in M$ and $\widetilde \gamma_0$ be a lift of $\gamma_0$ in
the universal cover of $M$.
Suppose $\gamma$ is a regular closed curve on $M$
which is freely homotopic to $\gamma_0$.
In the cases
\begin{itemize}
\item[(W1)] $M$ is orientable,
\item[(W2)] $M$ is non-orientable, and $\gamma_0$ is orientation reversing,
\item[(W3)] $M$ is non-orientable, and $\gamma_0$ is orientation preserving and reversible,
\end{itemize}
we define the ({\it free}) {\it winding number} $W_{\widetilde\gamma_0}(\gamma)$ of $\gamma$
to be equal to $W(\gamma)$.
In the case
\begin{itemize}
\item[(W4)] $M$ is non-orientable, and $\gamma_0$ is orientation preserving
and non-reversible,
\end{itemize}
we define $W_{\widetilde\gamma_0}(\gamma)$ of $\gamma$ as follows:
Let $H:\gamma_0\simeq\gamma$
be a free homotopy which keeps the curve closed.
Then $H$ lifts to a free homotopy $\widetilde H:\widetilde\gamma_0\simeq \widetilde\gamma$
from $\widetilde\gamma_0$ to some lift $\widetilde\gamma$ of $\gamma$.
Let $\widetilde p$ be the initial point $\widetilde\gamma(a)$ of $\widetilde\gamma$.
The ({\it free}) {\it winding number} $W_{\widetilde\gamma_0}(\gamma)$ of
the regular closed curve $\gamma$ above is defined to be $w_{\widetilde p}(\gamma)$.
\end{definition}

\begin{proposition} \label{thm:well-def}
$W_{\widetilde\gamma_0}(\gamma)$ is well-defined and is a regular homotopy invariant of $\gamma$.
\end{proposition}

\begin{proof}
We only need to consider the (W4) case.
Suppose there are two free homotopies $H:\gamma_0\simeq\gamma$ and $K:\gamma_0\simeq\gamma$.
These lift to homotopies $\widetilde H:\widetilde\gamma_0\simeq\widetilde\gamma$
and $\widetilde K:\widetilde\gamma_0\simeq\widehat\gamma$.
Let $\widetilde p$ and $\widehat p$ denote the initial points
$\widetilde\gamma(a)$ and $\widehat\gamma(a)$, respectively.
We want to show $w_{\widetilde p}(\gamma)=w_{\widehat p}(\gamma)$.
The composite of $H$ and $K$ determines a closed curve $\xi$ based at the base point $p$ of $\gamma$,
and the composed homotopy can be regarded as a homotopy between the `conjugate'
$\xi^*\gamma\xi$
and $\gamma$ which fixes the base point.  Since $\gamma$ is non-reversible,
$\xi$ must be orientation preserving.  
Therefore $\varepsilon(\widetilde p, \widehat p)=1$.
By, \ref{thm:wn-lifts}, we have $w_{\widetilde p}(\gamma)=w_{\widehat p}(\gamma)$.
This finishes the proof of the well-definedness.

Next suppose $\gamma$ and $\eta$ are regularly homotopic regular closed curves.
As in the definition above, let $\widetilde\gamma$ be the lift of $\gamma$ determined by 
the lift of a homotopy $H:\gamma_0\simeq \gamma$ which starts from $\widetilde\gamma_0$.
Then a free homotopy between $\gamma_0$ and $\eta$ is obtained by composing
$H$ and a regular homotopy $K$ between $\gamma$ and $\eta$.
Now the desired lift $\widetilde\eta$ which is used to define $W_{\widetilde \gamma_0}$
is determined by the lift $\widetilde K$ starting from $\widetilde\gamma$.  
By \ref{thm:wn-reg-h}, $w_{\widetilde p}(\gamma)=w_{\widetilde q}(\eta)$, where
$\widetilde p$ and $\widetilde q$ denotes the initial points of
$\widetilde\gamma$ and $\widetilde\eta$, respectively.
This completes the proof of the regular homotopy invariance.
\end{proof}


\noindent
{\bf Proof of Theorem \ref{thm:main2}.}

\noindent
We only need to prove that (2) implies (1).
So assume that $\gamma_1$ and $\gamma_2$ are both freely homotopic to $\gamma_0$ keeping the curve closed
and that their winding numbers $W_{\widetilde\gamma_0}(\gamma_1)$
and $W_{\widetilde\gamma_0}(\gamma_2)$ are the same.
Let $p_0$ be the base point of $\gamma_0$.
We can apply a finger move and a local rotation to $\gamma_2$ 
so that $\gamma_1$ and $\gamma_2$ have the same base point $p$ and the same base direction.
These operations are regular homotopies, so the winding numbers are unchanged.
Since they are freely homotopic through $\gamma_0$, 
there exists a smooth closed curve $\xi$ based at $p$ such that
if we regularly homotope $\gamma_2$ by a finger move along $\xi$
then $\gamma_1$ and the new $\gamma_2$ are homotopic fixing the base point.
Let $\widetilde p$ be a lift of $p$ in the universal cover $\widetilde M$, and
$\widetilde\gamma_1$ and $\widetilde\gamma_2$ be the lifts of $\gamma_1$ and $\gamma_2$
with the initial point $\widetilde p$, respectively.
Since $\gamma_j$'s are homotopic fixing the base point,
$\widetilde\gamma_j$'s have the same terminal point.
Note that, if $M$ is non-orientable and $\gamma_j$'s are reversible,
then we have $|w_{\widetilde p}(\gamma_1)|=|w_{\widetilde p}(\gamma_2)|$
and that, otherwise, we have $w_{\widetilde p}(\gamma_1)=w_{\widetilde p}(\gamma_2)$.
So suppose $M$ is non-orientable, $\gamma_j$'s are reversible,
and  $w_{\widetilde p}(\gamma_1)=-w_{\widetilde p}(\gamma_2)$.
Then a finger move of $\gamma_2$ along a suitably chosen closed loop
reverses the sign of $w_{\widetilde p}(\gamma_2)$ without changing the homotopy class;
so we may assume $w_{\widetilde p}(\gamma_1)=w_{\widetilde p}(\gamma_2)$ in every case.

Now, by \ref{thm:main1}, $\gamma_1$ and $\gamma_2$ are regularly homotopic fixing the base point.
This finishes the proof of Theorem \ref{thm:main2}.
\hfill $\square$

\bigbreak

\begin{example} \label{ex-trivial} Here are examples of regular closed curves
with trivial free winding numbers.
\begin{itemize}
\item[(1)] figure eight curves contained in discs
\item[(2)] geodesics that are regularly closed
\item[(3)] closed horocycles on hyperbolic surfaces (Figure 2 shows the case where
the lifts are part of a horizontal horocycle in the upper half plane model of $\bbH^2$.
Notations are the same as in \ref{def:ext-a}.)
\end{itemize}
\end{example}

\begin{figure}[htbp]
  \begin{center}
   \includegraphics[clip, width=60mm]{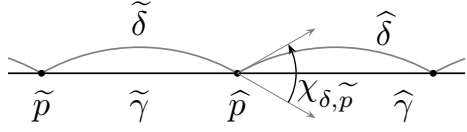}
  \end{center}
\caption{Lifts of $\gamma$ and $\delta$}\label{fig:lifts}
\end{figure}

\bigbreak

Finally, we check the dependence of the winding number on the choice of
$\widetilde\gamma_0$.

\begin{proposition} \label{thm:fwn-choice}
Assume that $M$ is non-orientable and that $\gamma_0$ is an orientation preserving
and non-reversible closed curve with base point $p_0$.
\par\noindent
{\rm (1)} Suppose that $\widetilde{p}_0$ and  $\widehat{p}_0$ are lifts of $p_0$ in $\widetilde M$, and 
$\widetilde\gamma_0$ and $\widehat{\gamma}_0$ are the lifts of $\gamma_0$ with the initial point $\widetilde p_0$ and $\widehat{p}_0$, respectively.
Then we have 
$W_{\widehat\gamma_0}(\gamma)=\varepsilon(\widetilde{p}_0, \widehat{p}_0)
W_{\widetilde\gamma_0}(\gamma).$
\par\noindent
{\rm (2)} Suppose that there is a free homotopy $K$ from another closed curve $\gamma_1$
to $\gamma_0$ which keeps the curve closed and  suppose that a lift $\widetilde K$ of $K$
is a homotopy from $\widetilde\gamma_1$ to $\widetilde\gamma_0$.
Then we have $W_{\widetilde\gamma_0}(\gamma)=W_{\widetilde\gamma_1}(\gamma).$
\end{proposition}

\begin{proof} Let $H$, $\widetilde H$, $\widetilde\gamma$ and $\widetilde p$ be 
as in \ref{def:fwn} (W4).
\par\noindent
(1) Let $T$ denote the deck transformation which sends $\widetilde p_0$ to $\widehat{p}_0$.
Then $T\circ\widetilde H$ is a homotopy between $\widehat\gamma_0=T\circ \widetilde\gamma_0$
and $T\circ\widetilde\gamma$; therefore, we have
\[
W_{\widehat\gamma_0}(\gamma)=w_{T(\widetilde p)}(\gamma)
\buildrel{\ref{thm:wn-lifts}}\over{=}\varepsilon(\widetilde p,T(\widetilde p))w_{\widetilde p}(\gamma)
=\varepsilon(\widetilde p_0,T(\widetilde p_0))w_{\widetilde p}(\gamma)
=\varepsilon(\widetilde p_0,\widehat p_0)W_{\widetilde\gamma_0}(\gamma).
\]
\par\noindent
(2) If we compose the two homotopies $\widetilde K$ and $\widetilde H$, then it is a 
homotopy from $\widetilde\gamma_1$ to $\widetilde\gamma$ and it covers the composite
of $K$ and $H$. Therefore, $W_{\widetilde\gamma_0}(\gamma)$
and $W_{\widehat\gamma_0}(\gamma)$ are both equal to $w_{\widetilde p}(\gamma)$.
\end{proof}

\bigbreak

\end{document}